\documentclass[12pt]{amsart}
\usepackage[T1]{fontenc}
\usepackage{amscd, amsfonts, amsmath, amssymb, amsthm, enumerate, flexisym, float, graphicx, hyperref, pdfpages,tikz, mathtools,comment, pinlabel, subcaption, tabu, tikz, tikz-cd, url, wasysym}
\usepackage[autostyle]{csquotes}
\usepackage[all,cmtip]{xy}
\usepackage{tikz-cd}
\usetikzlibrary{cd, arrows, calc, knots, patterns, positioning, shapes.callouts, shapes.geometric, shapes.symbols, decorations.markings, hobby}
\tikzset{>=stealth', arrow/.style={->}}
\makeatletter
\@namedef{subjclassname@2020}{Mathematics Subject Classification \textup{2020}}
\tikzset{
	oriented knot strand/.style={
		thick,
		only when rendering/.style={
			postaction={decorate},
			decoration={
				markings,
				mark=at position 0.16 with {\arrow{>}}
			}
		}
	},
	oriented curve/.style={
		thick,
		postaction={decorate},
		decoration={
			markings,
			mark=at position 0.16 with {\arrow{>}}
		}
	}
}
\newsavebox{\@brx}
\newcommand{\llangle}[1][]{\savebox{\@brx}{$\m@th{#1\langle}$}%
	\mathopen{\copy\@brx\kern-0.5\wd\@brx\usebox{\@brx}}}
\newcommand{\rrangle}[1][]{\savebox{\@brx}{$\m@th{#1\rangle}$}%
	\mathclose{\copy\@brx\kern-0.5\wd\@brx\usebox{\@brx}}}

\makeatother

\newtheorem{theorem}{Theorem}[section]
\newtheorem{corollary}[theorem]{Corollary}
\newtheorem{lemma}[theorem]{Lemma}
\newtheorem{proposition}[theorem]{Proposition}

\theoremstyle{definition}

\newtheorem{example}[theorem]{Example}
\newtheorem{remark}[theorem]{Remark}
\numberwithin{equation}{subsection}
\newtheorem*{ack}{Acknowledgement}
\newtheorem*{declare}{Declaration}

\newcommand{\Env}{\operatorname{Env}}
\newcommand{\Stab}{\operatorname{Stab}}

\newcommand{\HH}{\mathbb{H}}
\newcommand{\Aut}{\operatorname{Aut}}

\newcommand{\C}{\operatorname{C}}
\newcommand{\Inn}{\operatorname{Inn}}

\newcommand{\Hom}{\operatorname{Hom}}

\newcommand{\Isom}{\mathrm{Isom}}
\newcommand{\scl}{\mathrm{scl}}
\newcommand{\Map}{\mathrm{Map}}

\newcommand{\Ha}{\textbf{H}}
\newcommand{\Za}{\textbf{Z}}
\newcommand{\Ba}{\textbf{B}}
\newcommand{\Ca}{\textbf{C}}
\newcommand{\ca}{\textbf{c}}

\begin{document}
	\title{Bounded cohomology of fundamental quandles of links}
	\author{Deepanshi Saraf}
	\author{Mahender Singh}
	
	\address{Department of Mathematical Sciences, Indian Institute of Science Education and Research (IISER) Mohali, Sector 81,  S. A. S. Nagar, P. O. Manauli, Punjab 140306, India.}
	\email{saraf.deepanshi@gmail.com}
	\email{mahender@iisermohali.ac.in}

	
\begin{abstract}
We study bounded cohomology of quandles, a recently introduced framework bringing geometric methods into quandle theory. We establish two general criteria guaranteeing the infinite dimensionality of the second bounded cohomology of a quandle. The first is formulated in terms of homogeneous quasimorphisms on the inner automorphism group, while the second is based on the vanishing of stable commutator length on a suitable subgroup. As topological applications, we prove that, for every link other than the unknot and the Hopf link, the second bounded cohomology of its fundamental quandle has dimension equal to the cardinality of the continuum. We further consider fundamental $n$-quandles of links, which are  closely related to the $n$-fold cyclic branched covers of the 3-sphere. For every $n\geq 3$, we prove that the second bounded cohomology of the fundamental $n$-quandle of a hyperbolic knot other than the figure-eight knot has dimension equal to the cardinality of the continuum. As algebraic applications, we prove that the second bounded cohomology of a free product of quandles is infinite dimensional whenever the inner automorphism group of at least one free factor is amenable. In particular, this yields infinite-dimensionality of the second bounded cohomology for free quandles of rank greater than one and for their canonical quotients.
\end{abstract}	
\maketitle
	
\section{Introduction}
Racks and quandles are self-distributive algebraic structures, and their defining axioms correspond to algebraic representations of the Reidemeister moves for planar diagrams of knots and links in the 3-sphere. In addition to providing powerful invariants for knots and links, as shown by Joyce \cite{Joyce1979} and Matveev \cite{MR0672410}, these structures appear in a broad range of mathematics, including Hopf algebras \cite{MR1994219} and set-theoretic solutions to the Yang-Baxter equation \cite{MR1183474}, to name a few.
\par 

A (co)homology theory for racks and quandles was initially introduced by Fenn, Rourke and Sanderson \cite{MR1364012}, who approached it from a homotopy-theoretic perspective using classifying spaces, defining the homology of the rack space of a rack. Carter et al. \cite{MR1990571} applied rack and quandle cohomology to knot theory by constructing state sum invariants. These cohomology theories have since been generalised in \cite{MR1994219, MR3558231}. More recently, Szymik \cite{MR3937311} showed that quandle cohomology can be viewed as a form of Quillen cohomology.
\par

Bounded cohomology was introduced by Johnson in the context of Banach algebras \cite{MR0374934} and later extended by Gromov to topological spaces \cite{MR0686042}. It has since become a key tool in the geometry of manifolds \cite{MR0686042}, rigidity theory \cite{MR1911660}, and stable commutator length \cite{MR2527432}.  In a recent work, K\k{e}dra \cite{MR4779104} explored a natural metric on racks and quandles arising from their inner symmetries. He further developed a bounded cohomology theory for racks and quandles, and proved that a rack or quandle with finitely many connected components is bounded with respect to this metric if and only if the comparison map from its second bounded quandle cohomology to its second ordinary quandle cohomology is injective. Moreover, K\k{e}dra showed that the fundamental quandles of non-trivial knots, as well as free quandles, are unbounded. In particular, their second bounded quandle cohomology is non-trivial. Consequently, the second bounded cohomology of the fundamental quandle of a knot detects the unknot, thereby contributing to the ongoing search for knot invariants capable of distinguishing the unknot.
\par

On the other hand, Fujiwara \cite{MR1476898} proved that the second bounded group cohomology of a knot group is infinite-dimensional. This naturally raises the question of whether the second bounded quandle cohomology of the fundamental quandle of a knot is likewise infinite-dimensional. In this paper, we address this question for fundamental quandles of links, their canonical quotients, free products, and establish several  general results. Our findings indicate that second bounded rack and quandle cohomology exhibits a well-behaved structure. Nevertheless, a complete understanding of bounded cohomology in the setting of racks and quandles remains an interesting and challenging problem.
\par

The layout of the paper is as follows.
\par

In Section \ref{section prelim}, we review the basic concepts of quandle theory that will be necessary for our work. Section \ref{section quasimorphisms}  provides an overview of key results on quasimorphisms. In particular, we establish an isomorphism between a suitable quotient of the space of quandle quasimorphisms and the kernel of the comparison map from bounded quandle cohomology to ordinary quandle cohomology (Proposition \ref{kernel-qh quandle}).
\par
 Section \ref{section sec bounded cohomology} presents our main results. Exploiting the fact that a quandle $X$ can be written as a union $\sqcup_{i \in I}(G/H_i,z_i)$ of certain cosets in a group $G$, we prove three key results that provide sufficient conditions for the infinite dimensionality of the second bounded quandle cohomology of $X$. The first two results involve infinite dimensionality of the subspace of homogeneous group quasimorphisms of $G$ that vanish on a subgroup $H_{i_0}$ and on all but finitely many of the elements $z_i$ (Theorem \ref{Second bounded Cohomology - vanishing qm} and Theorem \ref{Second bounded Cohomology - vanishing qm 2}), while the third result involves the infinite dimensionality of the second bounded group cohomology of $G$ and the vanishing of the stable commutator length on a subgroup $H_{i_0}$ of $G$ (Theorem \ref{second bounded cohomology}).
\par

The final section of the paper, Section \ref{section link and free quandles}, presents applications of the preceding results to links, free quandles, and their canonical quotients. We show that the second bounded cohomology of the fundamental quandle of any oriented link in the 3-sphere has dimension equal to the cardinality of the continuum, except in the cases of the unknot and the Hopf link (Proposition \ref{second bounded cohomology of non-split link} and Proposition \ref{second bounded cohomology of split link}). Fundamental $n$-quandles of oriented links, which arise as canonical quotients of their fundamental quandles, are closely related to the $n$-fold cyclic branched covers of the 3-sphere branched over the corresponding links. Using geometric arguments, we prove that if $K$ is an oriented hyperbolic knot and $n\ge 3$ such that $(K,n)\neq (4_1,3)$, then the second bounded cohomology of its fundamental $n$-quandle has dimension equal to the cardinality of the continuum (Theorem \ref{thm:quandleinf}). We then show that the second bounded quandle cohomology of a free product of quandles is infinite-dimensional whenever the group of inner automorphisms of at least one of the free factors is amenable (Proposition \ref{second bounded cohomology of free product}). Finally, we establish an analogous result for free $k$-quandles of rank $n$ whenever $n \geq 2$ and $k \geq 3$ or $n \geq 3$ and $k =2$ (Proposition \ref{bounded cohom free n-quandle}). 
\medskip

\section{Preliminaries}\label{section prelim}
This section reviews the essential preliminaries that will be employed throughout the paper. 
\par

\subsection{Racks and quandles} To set our conventions, recall that a {\it quandle} is a set $X$ equipped with a binary operation $\ast$ that satisfies the following axioms:
	\begin{enumerate}
		\item $x \ast x = x$ for all $x \in X$.
		\item For each pair $x, y \in X$, there exists a unique $z \in X$ such that $x = z \ast y$.
		\item $(x \ast y) \ast z = (x \ast z) \ast (y \ast z)$ for all $x, y, z \in X$.
	\end{enumerate}

An algebraic structure satisfying the last two axioms is called a  {\it rack}. The second quandle axiom is equivalent to the existence of a dual binary operation on $X$, denoted by $(x, y) \mapsto x \ast^{-1} y$, satisfying $x \ast y = z$ if and only if $x = z \ast^{-1} y$ for all $x, y, z \in X$. For each $x \in X$, the map $S_x: X \rightarrow X$ given by $S_x(y) = y \ast x$ is an automorphism of $X$ that fixes $x$. The group $\Inn(X)$ generated by such automorphisms is called the {\it inner automorphism group} of $X$. 
\par
Given a quandle $X$, we define its {\it enveloping group} $\Env(X)$ with the presentation
$$ \Env(X)=\langle e_{x}, ~x \in X \mid e_{x \ast y}=e_{y}^{-1} e_{x} e_{y}~\textrm{for all}~x, y \in X \rangle.$$ 	
There is a right action of $\Env(X)$ on $X$, which on generators of  $\Env(X)$ is given by $$x \cdot e_y  = x*y$$ for $x, y \in X$. The quandle $X$ is called {\it connected} if $\Env(X)$ (equivalently $\Inn(X)$) acts transitively on $X$. In general, we denote a set of representatives of orbits under this action by $\mathcal{O}(X)$.
	
	\begin{example}
	Let $L$ be an oriented link in $\mathbb{S}^3$. In \cite[Section 4.5]{Joyce1979} and \cite[Section 6]{MR0672410}, Joyce and Matveev independently gave a topological construction of the  {\it fundamental quandle} $Q(L)$ of $L$,  and proved it to be an invariant of the isotopy type of  $L$. Further, they  proved that $Q(L)$ can also be obtained from a regular diagram $D$ of $L$. Suppose that $D$ has $s$ arcs and $t$ crossings. We assign labels $x_{1}, \ldots, x_{s}$ to the arcs of $D$, and then introduce the relation $r_l$ given by  $x_k*x_j=x_i$ or $x_k*^{-1}x_j=x_i$ at the $l$-th crossing of $D$ as shown in Figure \ref{link quandle}. It was proved that $$Q(L) \cong \langle x_1, \ldots, x_s \mid r_1,  \dots, r_t\rangle.$$
If $V(L)$ is a tubular neighborhood of $L$, then $\Env(Q(L))\cong \pi_1(\mathbb{S}^3\setminus V(L))$, the link group of $L$.

		\begin{figure}[hbt!]
			\begin{subfigure}{0.5\textwidth}
				\centering
				\begin{tikzpicture}[scale=0.8]
					\node at (0.6,-1.2) {{\small $x_k$}};
					\node at (-1.2,0.4) {{\small $x_j$}};
					\node at (1.5,1.2) {{\small $x_k \ast x_j = x_i$}};
					\begin{knot}[clip width=6, clip radius=6pt]
						\strand[thick, ->] (-2,0)--(2,0);
						\strand[thick, ->] (0,-2)--(0,2);
					\end{knot}
				\end{tikzpicture}
				\caption{Positive crossing}
			\end{subfigure}
			\begin{subfigure}{0.4\textwidth}
				\centering
				\begin{tikzpicture}[scale=0.8]
					\node at (1.8,1.2) {{\small $x_k \ast^{-1} x_j = x_i$}};
					\node at (-1.2,0.4) {{\small $x_j$}};
					\node at (0.6,-1.2) {{\small $x_k$}};
					\begin{knot}[clip width=6, clip radius=6pt]
						\strand[thick, ->] (2,0)--(-2,0);
						\strand[thick, ->] (0,-2)--(0,2);
					\end{knot}
				\end{tikzpicture}
				\caption{Negative crossing}
			\end{subfigure}
			\caption{Quandle relations at crossings.}
			\label{link quandle}
		\end{figure}
\end{example}

\begin{remark}
 It turns out that the fundamental quandles of the unknot and the Hopf link are trivial quandles with one and two elements, respectively (see Figure \ref{oriented Hopf link and unknot}). In fact, these are the only links with finite fundamental quandles. Here, a quandle $X$ is called {\it trivial} if $x*y=x$ for all $x, y \in X$. 

\begin{figure}[htbp]
	\centering
	
	\begin{subfigure}[b]{0.45\textwidth}
		\centering
		\begin{tikzpicture}[scale=0.75]
			
			\draw[
			thick,
			postaction={decorate},
			decoration={
				markings,
				mark=at position 0.30 with {\arrow{>}}
			}
			] (-0.6,0) circle[radius=1.5cm];
			
			\draw[
			thick,
			postaction={decorate},
			decoration={
				markings,
				mark=at position 0.05 with {\arrow{>}}
			}
			] (0.6,0) circle[radius=1.5cm];
			
			\draw[
			thick,
			preaction={draw=white,line width=5pt}
			]
			(-0.6,0) ++(62:1.5cm)
			arc[start angle=62,end angle=71,radius=1.5cm];
			
			\draw[
			thick,
			preaction={draw=white,line width=5pt}
			]
			(0.6,0) ++(242:1.5cm)
			arc[start angle=242,end angle=251,radius=1.5cm];
			
			\node at (-2.25,0.9) {$x_1$};
			\node at (2.25,0.9) {$x_2$};
			
		\end{tikzpicture}
		\caption{Oriented Hopf link.}
		\label{fig:hopf-link}
	\end{subfigure}
	\hfill
	\begin{subfigure}[b]{0.45\textwidth}
		\centering
		\begin{tikzpicture}[scale=0.75]
			
			\draw[
			thick,
			postaction={decorate},
			decoration={
				markings,
				mark=at position 0.16 with {\arrow{>}}
			}
			] (0,0) circle[radius=1.5cm];
			
			\node at (1.9,0.8) {$x_1$};
			
		\end{tikzpicture}
		\caption{Oriented unknot.}
		\label{fig:unknot}
	\end{subfigure}
	
	\caption{Oriented Hopf link and unknot.}
	\label{oriented Hopf link and unknot}
\end{figure}
 \end{remark}	
	
\begin{example}\label{example link quandle}
Although links in the 3-sphere offer a rich source of quandles, other interesting examples, especially those stemming from groups, also arise. We will rely heavily on one such construction.
Let $G$ be a group, $\{z_i \mid i \in I \}$ a set of elements of $G$, and $\{H_i \mid i \in I \}$ a set of subgroups of $G$ such that $H_i \le \C_G(z_i)$ for each $i$. Then we can define a quandle structure on the disjoint union $\sqcup_{i \in I} G/H_i$ of right cosets by
			$$H_i x \ast H_j y = H_i z_i^{-1}xy^{-1}z_j y.$$ 
We denote this quandle by $\sqcup_{i \in I} (G/H_i, z_i )$. Note that any quandle can be expressed as a quandle of this form \cite[Section 2.4]{Joyce1979}.

	\end{example}
\medskip

\subsection{Bounded cohomology of groups, racks and quandles}
Next, we recall the definition of bounded cohomology of groups \cite{MR3726870}, and that of racks and quandles \cite{MR4779104}. 
\par

We first recall the construction of the cochain complex defining the group cohomology with real coefficients \cite{MR1324339}. Let $G$ be a group. For each $n \geq 0$, let $\Ca^n(G, \mathbb{R})$ be the real vector space consisting of maps $G^n \to \mathbb{R}$, where we take $G^0$ to be the trivial group. The coboundary $$\partial^n:\Ca^n(G, \mathbb{R}) \to \Ca^{n+1}(G, \mathbb{R})$$ is given by
\begin{equation}\label{eqn1}
\partial^n(\sigma)(g_1,\ldots,g_{n+1}) = \sigma(g_2,\ldots,g_{n+1})  +  \sum_{i=1}^n (-1)^{i}\sigma(g_1,\ldots,g_ig_{i+1},...,g_{n+1}) + (-1)^{n+1} \sigma(g_1,\ldots,g_n),
\end{equation}
for all $\sigma \in \Ca^n(G, \mathbb{R})$ and $(g_1,\ldots,g_{n+1}) \in G^{n+1}$. The cohomology $\Ha^*(G,\mathbb{R})$ of the cochain complex $\Ca^*(G,\mathbb{R})$ is defined to be the {\it group cohomology} of $G$ with coefficients in $\mathbb{R}$. Let $\Ca^*_b(G, \mathbb{R})$ denote the subcomplex of bounded cochains with respect to the sup-norm. Then the cohomology $\Ha_b^*(G,\mathbb{R})$ of $\Ca^*_b(G,\mathbb{R})$ is called the {\it bounded group cohomology} of $G$ with coefficients in $\mathbb{R}$. We denote the group of bounded group cocycles and coboundaries by  $\Za_b^*(G,\mathbb{R})$ and $\Ba^*_b(G,\mathbb{R})$, respectively. The inclusion $\Ca^*_b(G,\mathbb{R}) \to \Ca^*(G, \mathbb{R})$ induces the {\it comparison} homomorphism $\ca^* : \Ha^*_b(G,\mathbb{R}) \to \Ha^*(G,\mathbb{R})$.
\par

Let $X$ be a rack and  $n \ge 0$ an integer. Let $C^n(X, \mathbb{R})$ be the real vector space consisting of maps $X^n \to \mathbb{R}$, where we understand $X^0$ to be a singleton set. The coboundary operator $\delta^{n-1} : C^{n-1}(X, \mathbb{R}) \to C^{n}(X, \mathbb{R})$ is defined by
	\begin{equation} \label{eq:coboundary_operator}
		\delta^{n-1} f(x_1, \ldots, x_n) = \sum_{i=1}^{n} (-1)^i \big(f(x_1, \ldots, x_{i-1}, x_{i+1}, \ldots, x_n) - f(x_1* x_i, \ldots, x_{i-1}* x_i, x_{i+1}, \ldots, x_n)\big)
	\end{equation}
for all $f \in C^{n-1}(X, \mathbb{R})$ and $x_i \in X$. A routine check shows that $C^*(X, \mathbb{R}) = \{C^n(X, \mathbb{R}), \delta^n\}$ is a cochain complex. The cohomology  $H^n(X, \mathbb{R})$ of this complex is called the {\it rack cohomology} of the rack $X$.
\par

If $X$ is a quandle, then we consider the subcomplex $C_Q^*(X,\mathbb R)$ of $C^*(X,\mathbb R)$ defined by
		\begin{eqnarray*}
			C_Q^n(X,\mathbb R)
			=
			\left\{
			f\in C^n(X,\mathbb R)
			\mid
			f(x_1,\ldots,x_n)=0
			\text{ whenever }x_i=x_{i+1}
			\text{ for some }1\leq i\leq n-1
			\right\}
		\end{eqnarray*}
		for $n\geq2$, and $	C_Q^n(X,\mathbb R)=C^n(X,\mathbb R)$
		for $n=0,1$. The cohomology of this subcomplex is called the {\it quandle cohomology} of $X$. We refer the reader to \cite[Section 3]{MR1990571} for details on quandle cohomology.
\par

The bounded cohomology of racks and quandles can be defined similarly to that of groups \cite{MR4779104}. Let $X$ be a rack. Then, for a map $f : X \to \mathbb{R}$, we write $\|f\|_\infty = \sup \{ |f(x)|~ \mid x \in X \}$. Let $C^n_{b}(X, \mathbb{R})$ be the subspace of $C^n(X, \mathbb{R})$ consisting of all maps $X^n \to \mathbb{R}$ which are bounded with respect to the sup-norm $\|\cdot\|_\infty$. The coboundary operators $\delta^n$ restrict to $C^n_{b}(X, \mathbb{R})$, and hence $C^*_{b}(X, \mathbb{R}) = \{C^n_{b}(X, \mathbb{R}), \delta^n\}$ forms a cochain complex. The cohomology  $H^n_b(X, \mathbb{R})$ of this complex is called the \textit{bounded rack cohomology}. The natural inclusion $C^n_{b}(X, \mathbb{R}) \hookrightarrow C^n(X, \mathbb{R})$ induces a homomorphism $c^n : H^n_{b}(X, \mathbb{R}) \to H^n(X, \mathbb{R})$, called the \textit{comparison map}. 
If $X$ is a quandle, we can similarly define the \textit{bounded quandle cohomology} and establish comparison maps. For simplicity, we denote the bounded cohomology of both the rack and the quandle as $H^n_b(X, \mathbb{R})$.
\par

Let $X$ be a rack or a quandle. K\k{e}dra defined a metric on each connected component of $X$ by 
	$$d(x,y)=\mathrm{min}\{n \in \mathbb{N} \mid y=(((x*^{\pm}x_1)*^{\pm}x_2)*^{\pm} \cdots )*^{\pm} x_n~\textrm{for some}~x_1, \ldots, x_n \in X\}.$$
We say that $X$ is {\it bounded} if the metric $d$ is bounded (has finite diameter) on each connected component of $X$. Otherwise, it is called {\it unbounded}. It turns out that fundamental quandles of non-trivial knots and free quandles are unbounded with respect to this metric \cite[Example 3.17]{MR4779104}. 
\par

A group $G$ is said to be {\it uniformly perfect} if there exists $N > 0$ such that any element of $G$ can be represented as a product of at most $N$ commutators. It is known from \cite[Corollary 2.11]{MR0787909} that if $G$ is uniformly perfect, then the comparison map $\ca^2: \Ha^2_b(G, \mathbb{R}) \to \Ha^2(G, \mathbb{R})$ is injective. For racks and quandles, the following result establishes a similar relation of this metric with the second bounded  cohomology \cite[Theorem 1.2]{MR4779104}.

\begin{theorem}\label{kedra main theorem}
Let $X$ be a rack or quandle with finitely many connected components. Then $X$ is bounded if and only if the kernel of the comparison map $c^2: H^2_b(X, \mathbb{R}) \to H^2(X, \mathbb{R})$ is trivial.
\end{theorem}
\medskip

\section{Quasimorphisms of groups, racks and quandles}\label{section quasimorphisms}
In this section, we revisit the concept of quasimorphisms for groups, racks, and quandles, along with their relationship to the corresponding bounded cohomology.

\subsection{Quasimorphisms of groups}  We refer to \cite{MR3726870} for more details on quasimorphisms of groups and the bounded group cohomology.
\par

A \textit{group quasimorphism} on a group $G$ is a map $f: G \to \mathbb{R}$ such that
\begin{equation} \label{eq:defect}
D(f) := \sup_{g, h \in G} |f(g) +f(h) -f(gh)| < \infty.
\end{equation}
The constant $D(f)$ is called the \textit{defect} of $f$. In terms of group cohomology, $f$ is a group quasimorphism if and only if $\partial^1(f)$ is bounded with respect to the sup-norm. We denote the vector space of all group quasimorphisms on $G$ by $\mathcal{G}(G,\mathbb{R}).$
\par	
A group quasimorphism $f: G \to \mathbb{R}$ is \textit{homogeneous} if it satisfies $f(g^n) = n \,f(g)$ for any $g \in G$ and $n \in \mathbb{Z}$. It is known that a homogeneous group quasimorphism is constant on conjugacy classes \cite[Subsection 2.2.3]{MR2527432}, that is,
	\begin{equation*} \label{eq:conjugacy}
		f(h^{-1} g h) =f(g)
	\end{equation*}
	for any $g, h \in G$.	The subspace of all homogeneous group quasimorphisms on $G$ is denoted by $\mathcal{HG}(G, \mathbb{R}).$
\par

Let $\ca^* : \Ha^*_b(G,\mathbb{R}) \to \Ha^*(G,\mathbb{R})$ be the comparison homomorphism. By \cite[Corollary 2.11]{MR3726870},  we have
	\begin{equation}
		\mathcal{G}(G, \mathbb{R}) = \mathcal{HG}(G, \mathbb{R}) \oplus \Ca^1_b(G, \mathbb{R})
	\end{equation}
		and
		\begin{equation}\label{ker-qh}
			\mathcal{HG}(G, \mathbb{R}) / \text{Hom}(G, \mathbb{R}) \cong \ker(\ca^2).
		\end{equation}
In fact, if $G$ is abelian, then \cite[Corollary 2.12]{MR3726870} gives
\begin{equation}\label{hom quasi hom equals hom}
\mathcal{HG}(G, \mathbb{R}) \cong \text{Hom}(G, \mathbb{R}).
\end{equation}
\medskip

\subsection{Quasimorphisms of racks and quandles} Next, we recall the definition of quasimorphism of racks and quandles \cite[Definition 5.1]{MR4779104}.
\par
A \textit{rack quasimorphism} on a rack $X$ is a map $f: X \to \mathbb{R}$ satisfying
\begin{equation} \label{eq:rack_quasimorphism}
\sup_{x,y \in X} |f(x) - f(x * y)| < \infty.
\end{equation}
\par

We denote the vector space of all rack quasimorphisms on a rack $X$ by $\mathcal{R}(X,\mathbb{R})$. As in the case of groups, $f$ is a rack quasimorphism if and only if $\delta^1(f)$ is bounded with respect to the sup-norm.  If $X$ is a quandle, then $f$ will be called a {\it quandle quasimorphism}, and we denote the corresponding vector space by $\mathcal{Q}(X,\mathbb{R})$. We shall use the following result \cite[Lemma 5.2]{MR4779104}.
	
\begin{lemma} \label{qh to bounded class}
Let $X$ be a rack (respectively, quandle). If $f : X \to \mathbb{R}$ is a rack (respectively, quandle) quasimorphism, then $\delta^1 (f)$ is a bounded rack (respectively, quandle) 2-cocycle such that its cohomology class $[\delta^1 (f)]$ lies in the kernel of the comparison map $c^2$. If $f$ is unbounded on a connected component of $X$, then $[\delta^1(f)]$ is non-trivial.
\end{lemma}

Given a rack $X$, we now determine the kernel of the comparison map $c^2: H^2_b(X, \mathbb{R}) \to H^2(X, \mathbb{R})$.	By definition, $\ker(c^2)$ consists of $ [\alpha] \in H^2_b(X, \mathbb{R})$ such that $ \alpha = \delta^1(f)$ for some $f \in C^1(X, \mathbb{R})$. As noted above, $\delta^1(f)= \alpha $ being bounded implies that $f$ is a quasimorphism. Thus, we have a map $\overline{\delta}^1 : \mathcal{R}(X, \mathbb{R}) \to \ker(c^2)$ given by $f \mapsto [\delta^1(f)]$.

\begin{proposition}\label{kernel-qh}
Let $X$ be a rack. Then the map $\overline{\delta}^1 : \mathcal{R}(X, \mathbb{R}) \to \ker(c^2)$ induces an isomorphism
		$$
		\mathcal{R}(X, \mathbb{R}) / \big( C^1_b(X, \mathbb{R}) + \Map(\mathcal{O}(X),\mathbb{R})\big) \cong \ker(c^2),
		$$
where $\Map(\mathcal{O}(X),\mathbb{R})$ is the space of maps $\mathcal{O}(X) \to \mathbb{R}$.
\end{proposition}

\begin{proof}
Clearly, $ \overline{\delta}^1 $ is a surjective linear map. 	Suppose that $f \in \mathcal{R}(X, \mathbb{R}) $ such that $ \overline{\delta}^1(f)= 0$ in $\ker(c^2) \subseteq H^2_b(X, \mathbb{R})$. Then there exists $ g \in C^1_b(X, \mathbb{R})$ such that $\delta^1(f)= \delta^1(g)$. It follows that $ \delta^1  (f - g) = 0 $, which gives $ (f - g) (x)= (f-g )(x*y)$ for all $x,y \in X$. Hence, $f-g \in \Map(\mathcal{O}(X), \mathbb{R})$, and consequently
$$ \ker(\overline{\delta}^1) \subseteq  C^1_b(X, \mathbb{R}) +  \Map(\mathcal{O}(X), \mathbb{R}).$$
The reverse inclusion is clear, and the proof is complete.
\end{proof}

We have a similar result for quandles.

\begin{proposition}\label{kernel-qh quandle}
Let $X$ be a quandle. Then the map $\overline{\delta}^1 : \mathcal{Q}(X, \mathbb{R}) \to \ker(c^2)$ induces an isomorphism
		$$
		\mathcal{Q}(X, \mathbb{R}) / \big( C^1_b(X, \mathbb{R}) + \Map(\mathcal{O}(X),\mathbb{R})\big) \cong \ker(c^2).
		$$
\end{proposition}

\begin{example}\label{example with unbounded amenable inn group}
Let $X=\mathbb{Z}$ be the permutation rack defined by $m*k=m+1$ for all $m, k \in X$. Then $(X, *)$ is a connected and unbounded rack.  For each $0<\alpha\leq 1$, define the map $f_\alpha: X \rightarrow \mathbb{R} $ by
	$$
	f_\alpha(m)=|m|^\alpha.
	$$
Then, the subadditivity of $f_\alpha$ gives
\begin{eqnarray*}
	\bigl|f_\alpha(m*k)-f_\alpha(m)\bigr| =	\bigl||m+1|^\alpha-|m|^\alpha\bigr|	\leq	\bigl||m+1|-|m|\bigr|^\alpha \leq
	1^\alpha=1
\end{eqnarray*}
for all $m, k \in X$. Let $0<\alpha_k < \cdots < \alpha_1 \le 1$ and $c_1, \ldots, c_k \in \mathbb{R}$ such that  $c_1 \ne 0$. If $g=\sum_{i=1}^k c_i f_{\alpha_i}$, then
$$g(m)=\sum_{i=1}^k c_i |m|^{\alpha_i} = c_1|m|^{\alpha_1}\left(1+\sum_{i =2}^k \frac{c_i}{c_1} |m|^{\alpha_i-\alpha_1}\right)\to c_1|m|^{\alpha_1}\to\pm\infty~\textrm{as}~m \to \infty.$$ 
Thus, the set $\{ f_\alpha \mid 0<\alpha\leq 1 \}$ is linearly independent modulo the space of bounded maps from $X$ to $\mathbb{R}$. Since $(0, 1]$ is uncountable, Proposition \ref{kernel-qh} implies that $\ker(c^2)$ has dimension equal to the cardinality of the continuum. Since the size of $C^2_b(X, \mathbb{R})$ is at most the cardinality of the continuum, it follow that the dimension of $H^2_b(X, \mathbb{R})$ is precisely the cardinality of the continuum.
\end{example}

\begin{remark}
{\rm
By \cite[Theorem 1.3]{MR4779104}, if $X$ is a rack with bounded and amenable inner automorphism group, then there is an isomorphism 
$$H_b^2(X, \mathbb{R}) \cong \Map_b(\mathcal{O}(X)^2,\mathbb{R}),$$
where $\Map_b(\mathcal{O}(X)^2,\mathbb{R})$ is the space of bounded maps  $\mathcal{O}(X)^2  \to \mathbb{R}$. The rack considered in Example \ref{example with unbounded amenable inn group} has the second bounded cohomology that is at least countably infinite-dimensional. Although the rack is connected and has amenable inner automorphism group, the group is unbounded. This shows that the boundedness condition on the inner automorphism group in \cite[Theorem 1.3]{MR4779104} is essential and cannot be omitted.}
\end{remark}

\medskip

\section{Second bounded cohomology of general quandles}\label{section sec bounded cohomology}
In this section, we examine the dimension of the second bounded cohomology of quandles. Recall from \cite[Section 2.4]{Joyce1979} that any quandle can be expressed as a quandle of the form $\sqcup_{i \in I}(G/H_i,z_i)$. Here, the group $G$ can be taken to be $\Inn(X)$, $\Aut(X)$ or $\Env(X)$ (see \cite[Proposition 2.3]{KapariSarafSingh2026}). Using this coset form, we prove that a homogeneous group quasimorphism gives rise to a quandle quasimorphism.

\begin{proposition}\label{quandle qh}
Let $X$ be a quandle such that $X=\sqcup_{i \in I}(G/H_i,z_i)$, where $G$ is a group, $\{z_i \mid i \in I\}$ is a set of elements of $G$ and $H_i \le \C_G(z_i)$ a subgroup for each $i \in I$. Let $\phi: G \rightarrow \mathbb{R}$ be a non-trivial homogeneous group quasimorphism on $G$ that vanishes on $H_{i_0} \cup \{z_i \mid i \in I\setminus I' \}$ for some fixed $i_0 \in I$ and a finite subset $I'$ of $I$. Let $\{x_j\}_{j \in J}$ be a set of right coset representatives of $H_{i_0}$ in $G$ such that $\phi(x_{j_0})\neq0$ for some ${j_0} \in J$. Then the map $\phi_X:X \rightarrow \mathbb{R}$ defined by $\phi_X(H_{i_0}a)=\phi(x_j)$, where $x_j$ is the unique right coset representative of $H_{i_0}a$ and  $\phi_X(H_ia)=0$ for all $i \neq i_0$, is an unbounded quandle quasimorphism.
\end{proposition}
\begin{proof}
For each fixed $n \in \mathbb{Z}$, we have $x_{j_0}^n=hx_k$ for some $k \in J$ and $h \in H_{i_0}$. Since $\phi$ is a homogeneous group quasimorphism, we obtain
$$|\phi_X(H_{i_0}x_k)|=|\phi(x_k)|=|\phi(h^{-1}x_{j_0}^n)|\geq |\phi(h^{-1})+\phi(x_{j_0}^n)|-D(\phi)\geq |n\phi(x_{j_0})|-D(\phi),$$ and hence $\phi_X$ is unbounded.
We note that $|\phi_X(H_{\ell} a)-\phi_X(H_{\ell}a*H_{r} b)|=0$ for all $\ell,r \in I$ such that $\ell\neq i_0$. For $\ell =i_0$ and any $r \in I$, 
let $x_i$ and $x_k$ be the right coset representatives of $H_{i_0}a$ and $H_{i_0}a*H_r b=H_{i_0}z_{i_0}^{-1}a b^{-1} z_r b$, respectively. Thus, we can write 
$a=ux_i$ and $z_{i_0}^{-1} a b^{-1} z_r b = v x_k$ for some  $u ,v\in H_{i_0}$. Using the fact that $\phi$ is a homogeneous group quasimorphism on $G$ that vanishes on $H_{i_0}$, we have
	\begin{eqnarray*}
		&& |\phi_X(H_{\ell} a)-\phi_X(H_{\ell}a*H_r b)|\\
		&=&|\phi_X(H_{i_0} a)-\phi_X(H_{i_0} z_{i_0}^{-1}ab^{-1}z_r b)|\\
		&=&|\phi(x_i)-\phi(x_k)|\\
		&=& |\phi(u^{-1}a)-\phi(v^{-1}z_{i_0}^{-1} a b^{-1} z_r b )|\\
		&=&|\phi(u^{-1}a)-\phi(v^{-1}z_{i_0}^{-1} a b^{-1} z_r b)+\phi(v^{-1}z_{i_0}^{-1} a b^{-1} z_r)-\phi(v^{-1}z_{i_0}^{-1} a b^{-1} z_r )+\phi(b)-\phi(b)|\\
		&\leq& |\phi(u^{-1}a)-\phi(v^{-1}z_{i_0}^{-1} a b^{-1} z_r)-\phi(b)|+D(\phi)\\
		&\leq& |\phi(u^{-1}a)+\phi(a)-\phi(a)-\phi(v^{-1}z_{i_0}^{-1} a b^{-1} z_r)+\phi(z_{i_0}^{-1} a b^{-1} z_r b)-\phi(z_{i_0}^{-1} a b^{-1} z_r )-\phi(b)|+D(\phi)\\
		&\leq& |\phi(a)-\phi(z_{i_0}^{-1} a b^{-1} z_r)-\phi(b)|+3D(\phi)\\
		&\leq& |\phi(a)-\phi(z_{i_0}^{-1} a b^{-1} z_r)+\phi(z_{i_0}^{-1})- \phi(z_{i_0}^{-1})+\phi( a b^{-1} z_r))-\phi( a b^{-1} z_r))-\phi(b)|+3D(\phi)\\
		&\leq& |\phi(a)+\phi(z_{i_0})-\phi( a b^{-1} z_r)+\phi(a b^{-1})-\phi(a b^{-1})+\phi(z_r)-\phi(z_r)-\phi(b)|+4D(\phi)\\
		&\leq& |\phi(a)+\phi(z_{i_0})-\phi(a b^{-1})-\phi(z_r)-\phi(b)|+5D(\phi)\\
		&\leq& |\phi(z_{i_0})-\phi(z_r)| +6D(\phi).
	\end{eqnarray*}
Since $\phi$ also vanishes on $\{z_i \mid i \in I\setminus I' \}$ for a finite subset $I'$ of $I$, it follows that $\sup \{|\phi_X(H_{\ell} a)-\phi_X(H_{\ell}a*H_r b)|\}$ is bounded. Hence $\phi_X$ is a quandle quasimorphism.
\end{proof}

\begin{theorem}\label{Second bounded Cohomology - vanishing qm}
Let $X$ be a quandle with finitely many connected components such that $X=\sqcup_{i \in I}(G/H_i,z_i)$, where $G$ is a group, $\{z_i \mid i \in I\}$ is a set of elements of $G$ and $H_i \le \C_G(z_i)$ a subgroup for each $i \in I$. Let $i_0 \in I$ and $I'$ a finite subset of $I$. Suppose that the subspace $\mathcal{HG}^{H_{i_0} \cup \{z_i \mid i \in I\setminus I' \}}(G, \mathbb{R})$ of homogeneous group quasimorphisms on $G$ that vanishes on $H_{i_0} \cup \{z_i \mid i \in I\setminus I' \}$ is infinite dimensional. Then $H^2_b(X,\mathbb{R})$ is infinite dimensional.
\end{theorem}

\begin{proof}
If $\phi \in \mathcal{HG}^{H_{i_0} \cup \{z_i \mid i \in I\setminus I' \}}(G, \mathbb{R})$, then by Proposition \ref{quandle qh}, $\phi_X$ is a quandle quasimorphism on $X$. Further, by Lemma \ref{qh to bounded class}, $[\delta^1(\phi_X)]$ is a bounded quandle $2$-cocycle.  This leads to a map
	\begin{eqnarray*}
		\Psi: \mathcal{HG}^{H_{i_0} \cup \{z_i \mid i \in I\setminus I' \}}(G, \mathbb{R}) \to H^2_b(X, \mathbb{R})
	\end{eqnarray*}
given by $\Psi(\phi)= [\delta^1(\phi_X)]$.	 If $\phi$ is trivial, then $\Psi(\phi)$ is well-defined by setting $\phi_X$ to be trivial. If $\phi$ is non-trivial, then the construction of $\phi_X$ depends on a choice of right coset representatives of $H_{i_0}$ in $G$. We claim that $\Psi(\phi)$ is independent of this choice. Let $\{x_i\}_{i \in I}$ and $\{y_i\}_{i \in I}$ be two sets of right coset representatives of $H_{i_0}$ in $G$. If necessary, after relabelling, for each $i \in I$, we have $y_i=h_i x_i$ for some $h_i \in H_{i_0}$. Define $\phi_X':X \to \mathbb{R}$ by $\phi_X'(Hy_i)=\phi(y_i)$ for each $i$. We need to show that $\delta^1(\phi_X-\phi_X')=\delta^1(\eta)$ for some $\eta \in C^1_b(X,\mathbb{R})$. To achieve this, define $\eta: X \rightarrow \mathbb{R}$ by $$\eta(H_{i_0}g)=\phi(x_k)-\phi(y_k)=\phi(x_k)-\phi(h_kx_k),$$ where $x_k, y_k$ are right coset representatives of $H_{i_0}g$ and $\eta(H_ig)=0$ for all $i \neq i_0$. Since $\phi$ vanishes on $H_{i_0}$, we see that $|\eta(H_{i}g)|\leq D(\phi)$ for all  $i \in I$, and hence $\eta \in C^1_b(X,\mathbb{R})$. Observe that for all $i \neq i_0$, we have  $\delta^1(\phi_X - \phi_X')(H_is, H_kt)= \phi_X(H_is) - \phi_X(H_is*H_kt) - \phi_X'(H_is) + \phi_X'(H_is*H_kt)=0$. And, for $i=i_{0}$, we have 
\begin{eqnarray*}
\delta^1(\phi_X - \phi_X')(H_{i_0}s, H_kt)&=& \phi_X(H_{i_0}s) - \phi_X(H_{i_0}s*H_kt) - \phi_X'(H_is) + \phi_X'(H_is*H_kt) \\
&=& \phi_X(H_{i_0}s) - \phi_X(H_{i_0}z_{i_0}^{-1}st^{-1}z_kt) - \phi_X'(H_{i_0}s) + \phi_X'(H_{i_0}z_{i_0}^{-1}st^{-1}z_kt) \\
&=& \phi(x_i) - \phi(x_j) - \phi(y_i) + \phi(y_j)\\
&=& \eta(H_{i_0}s) - \eta(H_{i_0}s* H_{k}t)\\
&=& \delta^1(\eta)(H_{i_0}s, H_{k}t),
\end{eqnarray*}
where $x_i, y_i$ are the coset representatives of $H_{i_0}s$ and  $x_j, y_j$ are the coset representatives of $ H_{i_0}z_{i_0}^{-1}st^{-1}z_kt$, respectively. Hence, $\delta^1(\phi_X-\phi_X')=\delta^1(\eta)$, and the map $\Psi$ is well-defined.
\par

Clearly, $\Psi$ is linear by construction. Suppose that $\phi \in \mathcal{HG}^{H_{i_0} \cup \{z_i \mid i \in I\setminus I' \}}(G, \mathbb{R})$ such that $ [\delta^1(\phi_X)] = 0$, that is, there exists $ \alpha \in C^1_b(X, \mathbb{R}) $ such that $ \delta^1 (\alpha) = \delta^1(\phi_X)$. Thus, the map $f =  \alpha -\phi_X : X \to \mathbb{R} $ is constant on connected components of $X$. Since $X$ has finitely many connected components, $f$ is bounded on $X$, and hence $ \phi_X$ is bounded. It follows from Proposition \ref{quandle qh} that $\phi$ must be a trivial quasimorphism. Hence, $\Psi$ is an injective linear map, and therefore $H^2_{b}(X, \mathbb{R})$ is infinite dimensional.
\end{proof}

By choosing the group $G$ appropriately, we obtain a version of Theorem \ref{Second bounded Cohomology - vanishing qm} that does not require the number of connected components of the quandle $X$ to be finite.

\begin{theorem}\label{Second bounded Cohomology - vanishing qm 2}
	Let $X$ be a quandle written as $X=\sqcup_{i\in I}(G/H_i,S_{x_i})$, where $G= \Inn(X)$, $I$ has cardinality equal to the number of connected components of $X$, $S_{x_i}\in \Inn(X)$ is the inner automorphism of $X$ induced by $x_i\in X$ and $H_i= \mathrm{Stab}_{\Inn(X)}(x_i)$ is the stabilizer of $x_i$ in $\Inn(X)$. Let $i_0 \in I$ and $I'$ a finite subset of $I$. Suppose that the subspace $\mathcal{HG}^{H_{i_0} \cup \{S_{x_i} \mid i \in I\setminus I' \}}(G, \mathbb{R})$ of homogeneous group quasimorphisms on $G$ that vanish on $H_{i_0} \cup \{S_{x_i} \mid i \in I\setminus I' \}$ is infinite dimensional. Then $H^2_b(X,\mathbb{R})$ is infinite dimensional.
\end{theorem}
\begin{proof}
	As in the proof of Theorem \ref{Second bounded Cohomology - vanishing qm}, we have the linear map
	\begin{eqnarray*}
		\Psi: \mathcal{HG}^{H_{i_0} \cup \{S_{x_i} \mid i \in I\setminus I' \}}(G, \mathbb{R}) \to H^2_b(X, \mathbb{R})
	\end{eqnarray*}
	given by $\Psi(\phi)= [\delta^1(\phi_X)]$. We claim that $\Psi$ is injective. Let $\phi \in \mathcal{HG}^{H_{i_0} \cup \{S_{x_i} \mid i \in I\setminus I' \}}(G, \mathbb{R})$ be such that $ [\delta^1(\phi_X)] = 0$, that is, there exists $ \alpha \in C^1_b(X, \mathbb{R}) $ such that $ \delta^1 (\alpha) = \delta^1(\phi_X)$. Thus, the map $f =  \alpha -\phi_X : X \to \mathbb{R} $ is constant on connected components of $X$. But, the connected components of $X$ are $G/H_i$, where $i \in I$. Let $f(G/H_i)= c_i$ for each $i \in I$. For $i\neq i_0$, by construction, we have  $\phi_X(x)=0$ for all $x \in G/H_i$. On the other hand, we have $c_{i_0}=f(x)=\alpha(x)-\phi_X(x)$ for all $x \in G/H_{i_0}$. This gives
	$$|\phi_X(x)| \leq \|\alpha\|_\infty+|c_{i_0}|$$
	for all  $x \in G/H_{i_0}$, and hence $\phi_X$ is bounded on all of $X$. Thus, Proposition \ref{quandle qh} implies that $\phi$ is a trivial quasimorphism. Hence, $\Psi$ is an injective linear map, and therefore $H^2_{b}(X, \mathbb{R})$ is infinite dimensional.
\end{proof}

\begin{remark}\label{cardinality continuum remark 1}
Let $X$ be a countable quandle. Then $H^2_b(X,\mathbb{R})$ has dimension at most the cardinality of the continuum. The proofs of both Theorems \ref{Second bounded Cohomology - vanishing qm} and  \ref{Second bounded Cohomology - vanishing qm 2} show that if  $\mathcal{HG}^{H_{i_0} \cup \{z_i \mid i \in I\setminus I' \}}(G, \mathbb{R})$ has dimension the cardinality of the continuum, then so does $H^2_b(X,\mathbb{R})$.
\end{remark}

\begin{lemma}\cite[Lemma 3.6]{MR1115747} or \cite[Lemma 2.24]{MR2527432} \label{defect of qi}
		Let $G$ be a group and $\phi$  a homogeneous group quasimorphism on $G$. Then there is an equality 
		$$
		D(\phi)=	\sup_{g,h \in G}|\phi([g,h])|.
		$$
\end{lemma}

Let $G$ be a group and $[G,G]$ its commutator subgroup. For $g \in [G, G]$, the {\it commutator length} $\mathrm{cl}(g)$ of $g$ is the least number of commutators whose product is $g$. The {\it stable commutator length} $\mathrm{scl}(g)$ is defined to be the limit 
$$\mathrm{scl}(g)= \lim_{n \mapsto \infty} \frac{\mathrm{cl}(g^n)}{n}.$$

Group quasimorphisms and stable commutator length are related by the following Bavard's duality theorem. See  \cite{MR1115747} or \cite[Theorem 2.70]{MR2527432}.

\begin{theorem}\label{Bavard's Duality Theorem}
Let $G$ be a group. Then, for any $g \in [G, G]$,  there is an equality
	$$
	\mathrm{scl}(g) = \frac{1}{2} ~\sup_{\phi \in \mathcal{HG}(G, \mathbb{R}) \setminus \Hom(G, \mathbb{R})} \frac{|\phi(g)|}{D(\phi)}.
	$$
\end{theorem}

It follows from Lemma \ref{defect of qi} and  Theorem \ref{Bavard's Duality Theorem} that the map $\mathrm{scl}$ is trivial on $[G,G]$ if and only if every homogeneous group quasimorphism has defect zero, that is, $ \mathcal{HG}(G, \mathbb{R}) = \Hom(G, \mathbb{R})$.

\begin{theorem}\label{second bounded cohomology}
Let $X$ be a quandle with finitely many connected components such that $X=\sqcup_{i \in I}(G/H_i,z_i)$, where $G$ is a group, $\{z_i \mid i \in I\}$ is a finite set of elements of $G$ and $H_i \le \C_G(z_i)$ a subgroup for each $i$. Suppose that $G$ is finitely presented with $\Ha^2_b(G,\mathbb{R})$ infinite dimensional, and there exists $i_0 \in I$ such that $H_{i_0}$ is finitely generated and $\mathrm{scl}(h)=0$ for every $h \in [H_{i_0},H_{i_0}]$. Then $H^2_b(X,\mathbb{R})$ is infinite dimensional.
\end{theorem}

\begin{proof}
	Since $G$ is a finitely presented group, it follows that $\Ha_n(G,\mathbb{Z})$ is a finitely generated abelian group for $n \leq 2.$ \cite[p.198, Exercise 1]{MR1324339}.
    The universal coefficient theorem gives the short exact sequence
    $$	0 \to \text{Ext}^{1}_{\mathbb{Z}}(\Ha_{n-1}(G, \mathbb{Z}), \mathbb{R}) \to \Ha^n(G, \mathbb{R}) \to \Hom(\Ha_n(G, \mathbb{Z}), \mathbb{R}) \to 0,
    $$
 which implies that $\Ha^2(G, \mathbb{R})$ is a finite dimensional vector space. But $\Ha^2_b(G,\mathbb{R})$ is given to be infinite dimensional, and therefore the kernel of the comparison map $\ca^2: \Ha^2_b(G,\mathbb{R}) \to \Ha^2(G,\mathbb{R})$ is infinite dimensional. Recall from \eqref{ker-qh} that $$ \mathcal{HG}(G, \mathbb{R}) / \Hom(G, \mathbb{R}) \cong \ker(\ca^2).$$  Since $G$ is finitely generated,  $\Hom(G, \mathbb{R})$ is finite dimensional, and hence $\mathcal{HG}(G, \mathbb{R})$ is infinite dimensional. By Lemma \ref{defect of qi} and Theorem \ref{Bavard's Duality Theorem}, since $\mathrm{scl}$ is trivial on $[H_{i_0}, H_{i_0}]$, it follows that $\mathcal{HG}(H_{i_0},\mathbb{R})=\Hom(H_{i_0},\mathbb{R}).$ Since $H_{i_0}$ is finitely generated,  the space $\Hom(H_{i_0},\mathbb{R})$ is finite dimensional.  Let us consider the restriction homomorphism $\theta: \mathcal{HG}(G, \mathbb{R}) \rightarrow  \mathcal{HG}(H_{i_0}, \mathbb{R}) = \Hom(H_{i_0},\mathbb{R})$ given by $f \mapsto f|_{H_{i_0}}.$ Thus, $ \ker(\theta)=\mathcal{HG}^{H_{i_0}}(G, \mathbb{R})$ is  infinite dimensional. The result now follows from Theorem \ref{Second bounded Cohomology - vanishing qm}.
\end{proof}

\begin{remark}\label{cardinality continuum remark 2}
If $X$ is a countable quandle, then the proof of Theorem \ref{second bounded cohomology} shows that if $\Ha^2_b(G,\mathbb{R})$ has dimension equal to the cardinality of the continuum, then so does $H^2_b(X,\mathbb{R})$.
\end{remark}

\begin{remark}
There are many groups with vanishing stable commutator length. For instance, amenable groups \cite[Proposition 2.65]{MR2527432}, uniformly perfect groups \cite[Corollary 2.11]{MR0787909} and groups obeying a law \cite[Main Theorem]{MR2580432} are examples of groups where the stable commutator length vanishes. Here, a group is {\it amenable} if it admits a finitely additive left-invariant probability measure. Further, a group $G$ is said to {\it obey a law} if there is a free group $F$ and a non-trivial element $w \in F$ such that $\phi(w)=1$ for every homomorphism $\phi:F  \to G$. For example, abelian groups are amenable and obey laws.
\end{remark}
\medskip

\section{Second bounded cohomology of fundamental quandles of links and free quandles}\label{section link and free quandles}
In this section, we give applications of  results from the preceding section to the second bounded cohomology of fundamental quandles of links, their canonical quotients, and free quandles. For notational convenience, we denote the free product of quandles $X_1, \ldots, X_n$ by $X_1 \star \cdots \star X_n$, and the free product of groups $G_1, \ldots, G_m$ by $G_1 \hexstar \cdots \hexstar G_m$. 

\subsection{Second bounded cohomology of fundamental quandles of links}
Recall that a connected 3-manifold is {\it irreducible} if every embedded 2-sphere bounds a 3-ball. Further, a compact orientable irreducible 3-manifold is called {\it Haken} if it contains a two-sided incompressible surface.  Let $L$ be an oriented link in $\mathbb{S}^3$ with component knots $K_1, \ldots, K_t$ and $V(L)= \cup_{i=1}^t V(K_i)$ be its tubular neighborhood. Let $X(L) := \mathbb{S}^3 \setminus V(L)$ and $\partial V(K_i)$ be the boundary of $V(K_i)$. For each $i$, we have $\pi_1(\partial V(K_i)) \cong \mathbb{Z} \oplus \mathbb{Z}$ and the inclusion map $\kappa_i : \partial V(K_i) \rightarrow X(L)$ induces a group homomorphism
$$
\kappa^*_i : \pi_1(\partial V(K_i)) \rightarrow \pi_1(X(L)).
$$
The group $H_i:=\kappa^*_i(\pi_1(\partial V(K_i)) )$ is called the $i$-th \textit{peripheral subgroup} of $\pi_1(X(L))$. First we consider the case of non-split links in the 3-sphere.

\begin{proposition} \label{second bounded cohomology of non-split link}
Let $L$ be an oriented non-split link other than the unknot and the Hopf link.  Then $H^2_{b}(Q(L), \mathbb{R})$ has dimension equal to the cardinality of the continuum.
\end{proposition}
\begin{proof}
It is known that the unknot and the Hopf link are the only links with solvable link groups \cite[p.6]{MR2931688}. By  \cite[Theorem 1.20]{MR3444187},  the fundamental group of a compact 3-manifold with a toroidal boundary is solvable if and only if it is virtually solvable. Thus, the link group $ \pi_1(X(L))$ is not virtually solvable. Since $L$ is a non-split link, $X(L)=\mathbb{S}^3 \setminus V(L)$ is irreducible and  a Haken 3-manifold \cite[p.57, C.22]{MR3444187}.  It now follows from \cite[Corollary 1.3]{MR1903743} that the dimension of $\Ha_b^2(\pi_1(X(L)), \mathbb{R})$ equals the cardinality of the continuum. Let  $K_1, \ldots, K_t$ be the component knots of $L$. As in the proof of \cite[Theorem 3.6]{MR4075375}, we can write $Q(L)\cong \sqcup_{i=1}^t \big( \pi_1(X(L))/H_i,m_i \big)$, where $m_i, l_i$ are the meridian  and the  longitude of the $i$-th component of $L$ and $H_i=\langle m_i,l_i \rangle$ is the $i$-th peripheral subgroup. Since the fundamental quandle  $Q(L)$ has finitely many connected components, $\pi_1(X(L))$ is a finitely presented group and $H_i$ is its finitely generated abelian subgroup for each $i$,  Theorem \ref{second bounded cohomology} yields the desired result.
\end{proof}

\begin{remark}
If $L$ is the Hopf link, then $Q(L)$ is the trivial quandle with two elements. By \cite[Example 4.4]{MR4779104}, $H^2_{b}(Q(L), \mathbb{R})$ is two-dimensional.
\end{remark}

\begin{corollary} \label{second bounded cohomology of knot quandle}
An oriented knot is trivial if and only if its second bounded quandle cohomology is trivial.
\end{corollary}
\begin{proof}
If $K$ is the unknot, then its fundamental quandle is the one-element trivial quandle, and hence $H_b^2(Q(K),\mathbb R)=0$. If $K$ is a non-trivial knot, then $K$ is non-split and is neither the unknot nor the Hopf link. Then, by Proposition  \ref{second bounded cohomology of non-split link}, $H_b^2(Q(K),\mathbb R)$ is non-trivial, which proves the assertion.
\end{proof}

Next, we consider the case of split links. Before proceeding, we make the following general observation.

\begin{lemma}\label{free product of quandles}
		Let $X_1, \ldots , X_m$ be quandles such that $X_i \cong \bigsqcup_{x_j \in \mathcal{O}(X_i)} (\Env(X_i)/H_j, e_{x_j})$ for each $i$, where $H_j$ is the stabilizer of $x_j$ in $\Env(X_i)$. Then 
		$$
		X_1 \star \cdots \star X_m \cong \bigsqcup_{z_k \in K} \big(\Env(X_1) \hexstar \cdots  \hexstar \Env(X_m)/H_k, z_k),
		$$
		where $K = \bigsqcup_{i=1}^m \mathcal{O}(X_i)$ and $z_k = e_{x_k} $ for $x_k \in K $.
\end{lemma}
\begin{proof}
It follows from \cite[Section 2.4]{Joyce1979} that $$X_1  \star \cdots \star X_m \cong \bigsqcup_{z_k \in K} (\Env(X_1 \star \cdots  \star X_m)/H_k, z_k),$$ where $z_k = e_{x_k} $ for $x_k \in K$ and $H_k=\Stab_{\Env(X_1 \star \cdots  \star X_m)}(x_k)$. By \cite[Lemma 4.1]{MR4075375}, we have $\Env(X_1 \star \cdots  \star X_m)=\Env(X_1) \hexstar \cdots  \hexstar \Env(X_m)$. If $g \in \Env(X_1) \hexstar \cdots  \hexstar \Env(X_m)$ is a reduced word such that $x_k \cdot g=x_k$ for $x_k \in \mathcal{O}(X_i)$, then we get the relation $g^{-1}e_{x_k}g=e_{x_k}$ in $\Env(X_1) \hexstar \cdots  \hexstar \Env(X_m)$. It follows from the normal-form in free products of groups that the centraliser of a non-trivial element of a free factor is contained in that free factor. Since $e_{x_k}\neq 1$, it follows that  $g$ must be a word in the generators of $\Env(X_i)$. Hence, for each $x_k \in \mathcal{O}(X_i)$, we have  $H_k=\Stab_{\Env(X_i)}(x_k)$, and the proof is complete.
\end{proof}

\begin{proposition} \label{second bounded cohomology of split link}
Let $L$ be an oriented split link with at least two non-split components. Then $H^2_{b}(Q(L), \mathbb{R})$ has dimension equal to the cardinality of the continuum.

\end{proposition}

\begin{proof}
Let $L_1, \ldots, L_k$ be the non-split components of $L$, where $k \ge 2$. Then we have $Q(L) \cong Q(L_1) \star \cdots \star Q(L_k)$. For each fixed $1 \le i \le k$, we have $Q(L_i) \cong \bigsqcup_{j=1}^{a_i} (\Env(Q(L_i))/ P_j,m_j)$, where $a_i$ is the number of component knots of $L_i$, $m_j$ is the meridian of the $j$-th component knot of $L_i$ and $P_j$ is its $j$-th peripheral subgroup. By  \cite[Lemma 4.1]{MR4075375}, we have $\Env(Q(L))\cong \Env(Q(L_1)) \hexstar \cdots  \hexstar \Env(Q(L_k))$. Note that, each group $\Env(Q(L_i))$ is  finitely presented, and hence so is $\Env(Q(L))$. Since all the free factors are infinite, it follows that $\Env(Q(L))$ is not isomorphic to $\mathbb Z_2\ast\mathbb Z_2$. Hence, by \cite[Corollary 1.1]{MR1491864}, $\Ha_b^2(G,\mathbb R)$ has dimension equal to the cardinality of the continuum. Each peripheral subgroup $P_j$ is finitely generated abelian, and hence the stable commutator length vanishes on  $[P_j,P_j]$. Since $Q(L)$ has finitely many connected components, Theorem \ref{second bounded cohomology} implies that that $H_b^2(Q(L),\mathbb R)$ is infinite-dimensional. Finally, since $Q(L)$ is countable, it follows from Remark \ref{cardinality continuum remark 2}  that  $H_b^2(Q(L),\mathbb R)$  has dimension equal to the cardinality of the continuum.
\end{proof}

Propositions \ref{second bounded cohomology of non-split link} and \ref{second bounded cohomology of split link} yield the following result.
\begin{corollary}
The fundamental quandle of an oriented link is non-trivial if and only if its second bounded cohomology has dimension equal to the cardinality of the continuum.
\end{corollary}
\medskip

\subsection{Second bounded cohomology of fundamental $n$-quandles of links}
Let $L$ be an oriented link in $\mathbb{S}^3$ and $Q(L)$ its fundamental quandle. For each $n\ge 2$, the \textit{fundamental $n$-quandle} of $L$ is defined as
$$Q_n(L) := Q(L)\big/ \langle x*^n y = x~\textrm{for all}~ x, y\in Q(L)\rangle.$$
Further, its \textit{$n$-enveloping group} is defined as
$$
\Env_n(Q_n(L)) = \langle \bold{e}_x,\ x\in Q_n(L) \mid \bold{e}_x^n=1\, \textrm{and} \, \bold{e}_{x*y}=\bold{e}_y^{-1}\bold{e}_x\bold{e}_y~\textrm{for all}~ x, y\in Q_n(L)\rangle.
$$

The following fact is well-known \cite[Corollary 5.1.11]{MR2634013}.

\begin{lemma}\label{lem: identification}
Let $K$ be a knot  and $n \ge 2$. Then $\Env_n(Q_n(K)) \cong \pi_1(\mathbb{S}^3\setminus V(K))/\langle\!\langle \mu^n\rangle\!\rangle$, where $\mu$ is a meridian of $K$ and $\langle\!\langle\,\cdot\,\rangle\!\rangle$ denotes the normal closure.
\end{lemma}

Though fundamental $n$-quandles are less sensitive invariants of oriented links, they are considerably more tractable and admit deep connections with $n$-fold cyclic branched covers of $\mathbb{S}^3$ branched over links.  Let $\widetilde{M}_n(K)$ be the $n$-fold cyclic branched cover of $\mathbb{S}^3$ for a knot $K$. The following result follows from \cite[Remark 5.1.5 and Theorem 5.2.2]{MR2634013}.

\begin{lemma}\label{lem:notsame}
Let $K$ be a knot  and $n \ge 2$. Then there is a short exact sequence of groups
$$
1 \longrightarrow \pi_1(\widetilde{M}_n(K)) \longrightarrow \Env_n(Q_n(K)) \xrightarrow{\ \varphi\ } \mathbb{Z}_n \longrightarrow 1,
$$
where $\varphi(\bold{e}_x)=1 \pmod n$ for each $x \in Q_n(K)$.
\end{lemma}

Thurston's orbifold geometrization theorem yields the following result \cite[Corollary 1.26]{MR1778789}.

\begin{theorem}\label{thm:paoluzzi}
Let $K$ be a hyperbolic knot and $n\ge 3$ such that $(K,n)\neq (4_1,3)$, where $4_1$ is the knot in Figure \ref{fig:figure-eight}. Then $\widetilde{M}_n(K)$ is a closed hyperbolic 3-manifold.
\end{theorem}

\begin{remark}
By \cite[Theorem~4.9]{MR0823282},  $\widetilde{M}_2(K)$ is a lens space for every $2$-bridge hyperbolic knot $K$, and hence it is not hyperbolic. Thus, Theorem~\ref{thm:paoluzzi} fails at $n=2$ for this infinite family of knots. Due to Dunbar's list of geometric orbifolds which are not hyperbolic \cite{MR0977042}, the figure-eight knot $4_1$ is the unique knot for which  $\widetilde{M}_3(4_1)$ is an euclidean manifold \cite[Corollary 1.26]{MR1778789}.

\begin{figure}[htbp]
	\centering
	\begin{tikzpicture}[
		use Hobby shortcut,
		scale=1.5
		]
		\begin{knot}[
			consider self intersections=true,
			ignore endpoint intersections=false,
			flip crossing=4,
			clip width=5
			]
			\strand[thick]
			([closed]0,0)
			.. (1.5,1)
			.. (0.5,2)
			.. (-0.5,1)
			.. (0.5,0)
			.. (0,-0.5)
			.. (-0.5,0)
			.. (0.5,1)
			.. (-0.5,2)
			.. (-1.5,1)
			.. (0,0);
		\end{knot}
		
	\path[
	draw=none,
	postaction={decorate},
	decoration={
		markings,
		mark=at position 0.06 with {\arrow{>}},
		mark=at position 0.32 with {\arrow{>}},
		mark=at position 0.50 with {\arrow{>}},
		mark=at position 0.80 with {\arrow{>}}
	}
	]
		([closed]0,0)
		.. (1.5,1)
		.. (0.5,2)
		.. (-0.5,1)
		.. (0.5,0)
		.. (0,-0.5)
		.. (-0.5,0)
		.. (0.5,1)
		.. (-0.5,2)
		.. (-1.5,1)
		.. (0,0);
	\end{tikzpicture}
	
	\caption{Oriented figure-eight knot $4_1$.}
	\label{fig:figure-eight}
\end{figure}

\end{remark}

By \cite[Chapter 2, Proposition 2.8]{MR2060653},  $$\pi_1(\mathbb{S}^3\setminus V(K))/\langle\!\langle\mu^n\rangle\!\rangle \cong \pi_1^{\mathrm{orb}}(O_n(K)),$$ where $\pi_1^{\mathrm{orb}}(O_n(K))$ is the \textit{orbifold fundamental group}  of the orbifold $O_n(K)$ whose underlying space is $\mathbb{S}^3$,  singular locus is the knot $K$, and the isotropy group along $K$ is the cyclic group $\mathbb Z_n$. We refer the reader to  \cite[Chapter 13]{Thurston-notes} for more details on orbifold theory.

\begin{lemma}\label{lem:transfer}
Let $K$ be a hyperbolic knot and $n\ge 3$ such that $(K,n)\neq (4_1,3)$. Then there is a discrete faithful representation
$$
\rho: \pi_1^{\mathrm{orb}}(O_n(K)) \longrightarrow \Isom(\HH^3).
$$
In particular, the orbifold $O_n(K)$ admits a complete hyperbolic structure.
\end{lemma}

\begin{proof}
By Theorem~\ref{thm:paoluzzi}, $\widetilde{M}_n(K)$ is a closed hyperbolic 3-manifold, and hence $\widetilde{M}_n(K)=\HH^3/\Gamma$ for a discrete, torsion-free and cocompact subgroup $\Gamma$ of $\Isom(\HH^3)$. In fact, we have $\Gamma\cong\pi_1(\widetilde{M}_n(K))$. By Lemma~\ref{lem:notsame}, we have a regular orbifold covering $\widetilde{M}_n(K)\to O_n(K)$ with the deck transformation group $\mathbb{Z}_n$ acting smoothly and faithfully on $\widetilde{M}_n(K)$. By Mostow--Prasad rigidity \cite{MR0385004,MR0385005}, each of these deck transformations is homotopic to a unique isometry. In fact, \cite[Theorem~H]{MR2491658} implies that this action is smoothly conjugate to an isometric action $\sigma$ of $\mathbb{Z}_n$ on $\widetilde{M}_n(K)$. Faithfulness of the smooth action implies the faithfulness of the isometric action $\sigma$. Moreover, the conjugation of two actions induces an orbifold diffeomorphism	$O_n(K)\cong	\widetilde{M}_n(K)/\sigma(\mathbb Z_n)$. Since $\widetilde{M}_n(K)=\HH^3/\Gamma$, there is a discrete extension $\widetilde\Gamma$ of $\Gamma$ by $\mathbb{Z}_n$ and an action $\rho: \widetilde\Gamma \longrightarrow \Isom(\HH^3)$ by isometries such that $\HH^3/\rho(\widetilde\Gamma) = O_n(K)$. By construction, we have $\rho(\widetilde\Gamma) \cong \pi_1^{\mathrm{orb}}(O_n(K))  \cong \Env_n(Q_n(K))$. In view of Lemma \ref{lem:notsame}, the action $\rho$ of $\widetilde\Gamma$ on $\HH^3$ is also faithful. Hence, the orbifold $O_n(K)$ admits a complete hyperbolic structure.
\end{proof}

Recall that the \textit{limit set} of a discrete subgroup $\Gamma$ of $\Isom(\HH^3)$ is the set 
$$\Lambda(\Gamma) = \overline{\Gamma x}\cap \partial\mathbb H^3,$$
 where $\Gamma x$ is the orbit of a point $x \in \HH^3$ and the closure is taken in the compactification $\overline{\mathbb H^3} = \mathbb H^3\cup\partial \mathbb H^3$. It turns out that this set does not depend on the choice of the point $x$. Further, $\Gamma$ is called \textit{elementary} if $\Lambda(\Gamma)$ contains at most two points.

\begin{corollary}\label{cor:cocompact}
The group $\Env_n(Q_n(K))$ acts properly discontinuously and cocompactly by isometries on $\HH^3$ and is non-elementary. Furthermore, $\Lambda(\Env_n(Q_n(K))) = \partial\HH^3$.
\end{corollary}

\begin{proof}
It is easy to see that if $G$ is a Lie group acting continuously and properly on a locally compact Hausdorff space and $H$ is a discrete subgroup of $G$, then $H$ acts properly discontinuously on that space. Note that $\Isom(\HH^3)$ is a Lie group acting continuously and properly on $\HH^3$. By Lemma \ref{lem:transfer}, $\Env_n(Q_n(K))$ is a discrete subgroup of $\Isom(\HH^3)$, and hence it acts properly discontinuously on $\HH^3$. Since $\HH^3/\Env_n(Q_n(K)) = O_n(K)$ is compact, the action is cocompact. By \cite[Exercise 8.2.10]{MR4221225}, elementary groups cannot act cocompactly with positive-volume quotient, and hence the group $\Env_n(Q_n(K))$ is non-elementary. Finally, by \cite[Section 5.1]{MR0698777}, a cocompact action has limit set equal to the entire boundary.
\end{proof}

We need the following result of Fujiwara \cite[p.~70]{MR1476898}.

\begin{theorem}\label{thm:fujiwara}
Let $G$ be a group acting properly discontinuously by isometries on a Gromov-hyperbolic space $X$. If the limit set $\Lambda(G)$ has at least three points, then $H^2_b(G,\mathbb{R})$ is infinite-dimensional with dimension equal to the cardinality of the continuum.
\end{theorem}

We next establish our main result concerning fundamental $n$-quandles of links.

\begin{theorem}\label{thm:quandleinf}
Let $K$ be an oriented hyperbolic knot and $n\ge 3$ such that $(K,n)\neq (4_1,3)$. Then $H^2_b(Q_n(K),\mathbb{R})$ has dimension equal to the cardinality of the continuum.
\end{theorem}

\begin{proof}
By Corollary~\ref{cor:cocompact}, the limit set $\Lambda(\Env_n(Q_n(K)))= \mathbb{S}^2$ is infinite and $\Env_n(Q_n(K))$ acts properly discontinuously and cocompactly by isometries on the Gromov-hyperbolic space $\HH^3$. Hence, by Theorem \ref{thm:fujiwara}, $\Ha^2_b(\Env_n(Q_n(K)),\mathbb{R})$ has dimension equal to the cardinality of the continuum. We write  $Q_n(K) \cong (\Env_n(Q_n(K))/\bar H,\bar z)$, where $\bar H$ is the image in $\Env_n(Q_n(K))$ of the peripheral subgroup $H$ of $\pi_1(\mathbb{S}^3\setminus V(K))$ \cite{MR3704243}. By Lemma \ref{lem: identification}, $\Env_n(Q_n(K))$ is finitely presented. Since $\bar H$ is finitely generated abelian,  $\scl$ vanishes on $[\bar H,\bar H]=\{1\}$. By Theorem \ref{second bounded cohomology}, $H^2_b(Q_n(K),\mathbb{R})$ is infinite-dimensional. Finally, it follows from Remark \ref{cardinality continuum remark 2} that $H^2_b(Q_n(K),\mathbb{R})$ has dimension equal to the cardinality of the continuum.
\end{proof}

\begin{remark}
Hoste and Shanahan \cite{MR3704243} proved that the fundamental $n$-quandle of an  oriented  link $L$ in the 3-sphere is finite if and only if the fundamental group $\pi_{1}(\widetilde{M}_{n}(L))$ of the $n$-fold cyclic branched cover of the 3-sphere is finite. In fact, using Thurston's geometrization theorem and Dunbar's classification of spherical 3-orbifolds, they derived the complete list of oriented links which have a finite fundamental $n$-quandle for some $n \ge 2$. It turns out that, for most links, the fundamental $n$-quandle is infinite for all but finitely many values of $n$.
\end{remark}

\begin{proposition}
Let $L$ be an oriented link with $k \ge 1$ component knots such that $Q_n(L)$ is finite for some $n \ge 2$. Then $H^2_b(Q_n(L),\mathbb{R})$ is $k(k-1)$-dimensional. In particular, if $L$ is a knot such that $Q_n(L)$ is finite, then $H^2_b(Q_n(L),\mathbb{R})$ is trivial.
\end{proposition}

\begin{proof}
For finite quandles, bounded cohomology coincides with ordinary cohomology. By \cite{MR1948837}, the dimension of the real second quandle cohomology of a finite quandle with $k$ connected components is $k(k-1)$. Since $Q_n(L)$ has $k$ connected components, it follows that $H^2_b(Q_n(L),\mathbb{R})$ is $k(k-1)$-dimensional. 
\end{proof}
\medskip

\subsection{Second bounded cohomology of free products}

\begin{proposition}\label{second bounded cohomology of free product}
Let $m\geq 2$, and let $X_1,\ldots,X_m$ be quandles, each having finitely many connected components. Suppose that $\Env(X_1)\hexstar\cdots \hexstar\Env(X_m)$ is finitely presented.
Assume further that $\Inn(X_i)$ is amenable and that $H_0=\Stab_{\Env(X_i)}(x_0)$ is finitely generated for some $x_0\in \mathcal{O}(X_i)$ and some $i$. Then $H_b^2(X_1\star\cdots\star X_m,\mathbb R)$ has dimension equal to the cardinality of the continuum.
\end{proposition}

\begin{proof}
For each $1 \le i\le m$, we write  $X_i \cong \bigsqcup_{x_j \in \mathcal{O}(X_i)} (\Env(X_i)/H_j, e_{x_j})$, where $H_j $ is the stabilizer of $x_j$ in $\Env(X_i)$. By Lemma \ref{free product of quandles}, we can write 
\begin{equation}\label{free product equation}
X_1 \star  \cdots \star X_m \cong \bigsqcup_{z_k \in K} (\Env(X_1) \hexstar  \cdots \hexstar \Env(X_m)/H_k, z_k),
\end{equation}
where $K = \bigsqcup_{i=1}^m \mathcal{O}(X_i)$, $z_k = e_{x_k} $ for $x_k \in K $. Since $\Env(X_i)$'s are always infinite, it follows from \cite[Corollary 1.1]{MR1491864} that $\Ha^2_b(\Env(X_1) \hexstar  \cdots \hexstar \Env(X_m), \mathbb{R})$ has dimension equal to the cardinality of the continuum. Without loss of generality, we assume that $\Inn(X_1)$ is amenable. Note that there is a central extension
$$1 \longrightarrow \ker(\pi) \longrightarrow \Env(X_1) \xrightarrow{\pi} \Inn(X_1) \longrightarrow 1,$$
where $\pi$ is given by $\pi(e_x)=S_x$ for each $x \in X_1$. Restricting the central extension to $H_0$, we obtain an extension
		$$
		1 \longrightarrow \ker(\pi) \cap H_0 \longrightarrow H_0 \xrightarrow{\pi|_{H_0}} \pi(H_0) \longrightarrow 1.
		$$
By \cite[Proposition 3.4]{MR3726870}, subgroups and extensions of amenable groups are amenable. Hence, $H_0$ is amenable, and therefore the stable commutator length vanishes on $[H_0,H_0]$. Since $\Env(X_1)\hexstar\cdots \hexstar\Env(X_m)$ is finitely presented and $H_0$ is finitely generated,  Theorem \ref{second bounded cohomology} 
implies that $H_b^2(X_1\star\cdots\star X_m,\mathbb R)$ is infinite dimensional. The identification given by \eqref{free product equation} implies that $X_1\star\cdots\star X_m$ is  countable. It now follows from Remark \ref{cardinality continuum remark 2} that $H_b^2(X_1\star\cdots\star X_m,\mathbb R)$ has dimension equal to the cardinality of the continuum.
\end{proof}

As an application of Theorem \ref{second bounded cohomology of free product}, we recover the following main result of  Aoki \cite[Theorem 1]{arXiv:2404.14752}.

\begin{corollary}\label{free quandle}
The second bounded cohomology of a non-trivial free quandle of finite rank has dimension equal to the cardinality of the continuum.
\end{corollary}

For a fixed $k \ge 2$, a quandle $X$ is called a {\it $k$-quandle} if each $S_x:X \to X$ has order dividing $k$. In other words, if $X$ is a $k$-quandle, then 
$$x \ast^{k} y:=(((x*\underbrace{y )*y)*\cdots )*y}_{k ~\mathrm{times}}=x$$ for all $x,y \in X$. If $X$ is a $k$-quandle, then we define its {\it $k$-enveloping group} by
$$\Env_k(X)=\langle e_{x}, ~x \in X \mid e_x^k=1,~e_{x \ast y}=e_{y}^{-1} e_{x} e_{y}~\textrm{for all}~x, y \in X \rangle.$$ 	
There is a surjective group homomorphism $\Env_k(X) \to \Inn(X)$ given by $e_x \mapsto S_x$.

\begin{proposition}\label{bounded cohom free n-quandle}
The second bounded cohomology of a free $k$-quandle of rank $n$ has	dimension equal to the cardinality of the continuum whenever $n\geq 2$ and $k\geq 3$, or $n\geq 3$ and $k=2$.
\end{proposition}

\begin{proof}
Let $S=\{a_1, \ldots, a_n\}$ be a set with $n \ge 2$. By  \cite[Corollary 2, Section 2.11]{Joyce1979}, the free $k$-quandle $FQ_k(S)$ on $S$ is the set of conjugates of elements of $S$  in the group $F_k(S) = \underbrace{\mathbb{Z}_k  \hexstar \cdots \hexstar \mathbb{Z}_k}_{n \text{ times}}$. Thus, we can write $FQ_k(S) \cong \sqcup_{i=1}^n(F_k(S)/\C_{F_k(S)}(a_i),a_i)$. It follows from \cite[Corollary 1.1]{MR1491864} that $\Ha_b^2(F_k(S),\mathbb R)$ has dimension equal to the cardinality of the continuum whenever $n\geq 2$ and $k\geq 3$, or $n\geq 3$ and $k=2$. Further, for each $i$, we have $\C_{F_k(S)}(a_i) \cong \C_{\mathbb{Z}_k}(a_i) \cong \mathbb{Z}_k$.  Note that, the group $F_k(S)$ is finitely presented, and the quandle $FQ_k(S)$ is countable. Thus, Theorem \ref{second bounded cohomology} and Remark \ref{cardinality continuum remark 2} imply that $H_b^2(FQ_k(S),\mathbb R)$ has dimension equal to the cardinality of the continuum.
\end{proof}

\begin{remark}
Let $S$ be a set of  two elements and $FQ_2(S)$ be the free 2-quandle on $S$. Note that $FQ_2(S)$ is a bounded quandle. We see that $\Env_2(FQ_2(S))\cong \mathbb{Z}_2 \hexstar \mathbb{Z}_2$, the infinite dihedral group. By \cite[Example 2.8]{MR4649062}, $\Env_2(FQ_2(S))$ is bounded. Further, by \cite[Lemma 2.9(a)]{MR4649062},  $\Inn(FQ_2(S))$ is also bounded being a quotient of $\Env_2(FQ_2(S))$. Since $\Env_2(FQ_2(S))$ is amenable, $\Inn(FQ_2(S))$ is also amenable. Finally, it follows from \cite[Theorem 1.3]{MR4779104} that the second bounded cohomology of $FQ_2(S)$ is finite dimensional.
\end{remark}

\begin{remark}\label{final remark}
The quandles considered in propositions \ref{second bounded cohomology of non-split link}, \ref{second bounded cohomology of split link},  \ref{second bounded cohomology of free product} and  \ref{bounded cohom free n-quandle} are all unbounded, and have infinite dimensional second bounded cohomology. Thus, it is plausible to ask whether there exists an unbounded quandle such that its second bounded cohomology is non-trivial but finite dimensional.
\end{remark}
\medskip

\begin{ack}
The authors also thank Prof. Jarek K\k{e}dra for his interest in the work and for pointing out an oversight in Remark \ref{final remark}. DS thanks IISER Mohali for the PhD research fellowship. MS acknowledges support from the ARG-MATRICS grant under Grant No. ANRF/ARGM/2025/000406/MTR.
\end{ack}

\begin{declare}
The authors declare that they have no conflicts of interest and that there is no data associated with this paper.
\end{declare}

\end{document}